\pdfoutput=1
\RequirePackage{ifpdf}
\ifpdf
\documentclass[pdftex]{sigma}
\else
\documentclass{sigma}
\fi

\numberwithin{equation}{section}

\newtheorem{thm}{Theorem}[section]

\newtheorem{lem}[thm]{Lemma}
 { \theoremstyle{definition}
\newtheorem{Remark}[thm]{Remark} }

\def\tr{\operatorname{tr}}
\def\bC{\mathbf{\overline{C}}}
\def\C{\mathbf{C}}

\def\lc{,\allowbreak\ldots,\allowbreak}
\let\ds\displaystyle
\def\Res{\mathop{\mathrm{Res}}\limits}

\def\yt{\tilde y}

\def\Wr{\mathop{\mathrm{Wr\vrule height0pt width.6pt}}\nolimits}

\begin{document}

\newcommand{\arXivNumber}{1801.08529}

\renewcommand{\PaperNumber}{058}

\FirstPageHeading

\ShortArticleName{Fuchsian Equations with Three Non-Apparent Singularities}

\ArticleName{Fuchsian Equations with Three Non-Apparent\\ Singularities}

\Author{Alexandre EREMENKO~$^\dag$ and Vitaly TARASOV~$^{\ddag\S}$}
\AuthorNameForHeading{A.~Eremenko and V.~Tarasov}
\Address{$^\dag$~Purdue University, West Lafayette, IN 47907, USA}
\EmailD{\href{mailto:eremenkos@math.purdue.edu}{eremenko@math.purdue.edu}}
\URLaddressD{\url{http://www.math.purdue.edu/~eremenko/}}
\Address{$^\ddag$~Indiana University -- Purdue University Indianapolis, Indianapolis, IN 46202, USA}
\EmailD{\href{mailto:vtarasov@iupui.edu}{vtarasov@iupui.edu}}
\Address{$^\S$~St.~Petersburg Branch of Steklov Mathematical Institute, St.~Petersburg, 191023, Russia}
\EmailD{\href{mailto:vt@pdmi.ras.ru}{vt@pdmi.ras.ru}}

\ArticleDates{Received February 02, 2018, in final form June 10, 2018; Published online June 15, 2018}

\Abstract{We show that for every second order Fuchsian linear differential equation $E$ with~$n$ singularities of which $n-3$ are apparent there exists a hypergeometric equation~$H$ and a linear differential operator with polynomial coefficients which maps the space of solutions of~$H$ into the space of solutions of~$E$. This map is surjective for generic parameters. This justifies one statement of Klein (1905). We also count the number of such equations~$E$ with prescribed singularities and exponents. We apply these results to the description of conformal metrics of curvature~$1$ on the punctured sphere with conic singularities, all but three of them having integer angles.}


\Classification{34M03; 34M35; 57M50}

\section{Introduction}

We consider Fuchsian linear differential equations of second order such that all but three singularities are apparent. A singularity is called apparent if the ratio of two linearly independent solutions is meromorphic at the singularity. Our main result says that all such equations can be obtained from hypergeometric equations by certain linear differential operators. The result was hinted by Klein at the end of his lecture on November~29, 1905 \cite{K}, without a proof or a precise formulation. F.~Shilling \cite{S} investigated the case of one apparent singularity with exponent difference~$2$.

Our motivation is the following. In \cite{E,EGT1,EGT}, we studied conformal Riemannian metrics of constant curvature $1$ on the punctured sphere with conic singularities at the punctures. The question is how many such metrics with prescribed singularities and prescribed conic angles at the singularities exist, and how to describe all of them. Unlike the similar problem with non-positive curvature where a complete answer is known \cite{H, P,T}, this problem is wide open, see, for example the survey in the introduction to~\cite{EGT1}. The developing map of such a metric is a ratio of two linearly independent solutions of a Fuchsian differential equation (see Section~\ref{section6}). Singularities of this equation are the conic singularities of the metric, and the exponent difference $\alpha$ at a singularity corresponds to the conic angle $2\pi\alpha$ of the metric. Another condition on this equation is that the projective monodromy group is conjugate to a subgroup of ${\rm PSU}(2)={\rm SU}(2)/\{\pm I\}$. We call such projective monodromy groups {\em unitarizable}.

So the question about the metrics is equivalent to the following:

{\em For given $a_j\in\bC$ and positive $\alpha_j$, $1\leq j\leq n$, how many Fuchsian differential equations of the form
\begin{gather}\label{11}
w''+\biggl(\sum_{j=1}^{n-1}\frac{1-\alpha_j}{z-a_j}\biggr)w'+\frac{Az^{n-3}+\cdots}{(z-a_1)\cdots(z-a_{n-1})}w=0
\end{gather}
with unitarizable projective monodromy exist?}

Here we assume without loss of generality that $a_n=\infty$ and the prescribed exponent difference at $\infty$ is $\alpha_n$; together with the Fuchs relation this determines the constant $A$. The dots in the numerator of the third summand of \eqref{11} stand for a polynomial of degree $n-4$ whose $n-3$ coefficients are called the {\em accessory parameters}. These accessory parameters must be determined from the condition that the projective monodromy is unitarizable. This gives a~system of equations on these parameters, and the question is how many solutions it has.

The difficulty of this problem strongly depends on the number of non-integer\footnote{Here and in what follows ``integer angle'' means an integer multiple of $2\pi$ radians.} angles $\alpha_j$. The condition of unitary projective monodromy implies that at a singular point with integer exponent difference the projective monodromy must be trivial, that is this singular point must be apparent.

A complete solution of this problem is only known when $n=3$, \cite{E, FKKRUY}. The case when $n$ is arbitrary and all angles are integers was studied in \cite{Sch}, and the case where there are only two non-integer
angles in~\cite{EGT1}. The case when $n=4$ and there are three non-integer angles was investigated in~\cite{EGT} under the additional assumption that equation \eqref{11} is {\em real}. In terms of the metric this means that the metric is symmetric with respect to some circle which contains all singularities. The problem of description of symmetric metrics is equivalent to the problem of classification of spherical $n$-gons of prescribed conformal type with prescribed angles, and this was the main subject of Klein's lectures~\cite{K}.

Theorem~\ref{theorem3} of this paper gives a classification of metrics of constant positive curvature with conic singularities on the sphere in the case when only three singularities have non-integer angles, for generic singularities and generic non-integer angles.

The word {\em generic} is used in this paper in the sense of Zariski topology: a set in $\C^n$ is generic if its complement is contained is a proper algebraic subset.

\section{Statement of results}\label{section2}

Let $a_j$ be distinct complex numbers, $a_j\not\in\{0,1\}$ for $1\leq j\leq k$. Let $m_j$ be positive integers, and $\alpha$, $\beta$, $\gamma$, $\delta$ complex numbers. Consider the differential equation
\begin{gather}\label{1a}
y''-\left( \frac\alpha z+\frac\beta{z-1}+\sum_{i=1}^k \frac{m_i}{z-a_i}\right) y'+\frac{\gamma\delta z^k+\cdots{}}{z(z-1)\prod\limits_{i=1}^k(z-a_i)}y = 0,
\end{gather}
where $\gamma+\delta=\alpha+\beta+1+\sum\limits_{j=1}^km_j$ (Fuchs's condition). The Riemann scheme of this equation is
\begin{gather*}
\left(\begin{matrix}
0&1&a_1&\ldots&a_k&\infty\\
0&0&0&\ldots&0&-\gamma\\
\alpha+1&\beta+1&m_1+1&\ldots&m_k+1&-\delta
\end{matrix}\right).
\end{gather*}
We recall that the first line of the Riemann scheme contains the singularities, and the other two the exponents at each singularity~\cite{Mellin}. Set $d=\sum\limits_{j=1}^km_j$. Assume that the singularities at \smash{$a_1\lc a_k$} are all apparent and
\begin{gather}\label{0d}
\gamma, \delta, \gamma-\alpha-1, \delta-\alpha-1 \not\in \{0,1\lc d-1\}.
\end{gather}

\begin{thm}\label{theorem1}
Under these conditions, there exists a polynomial $Q$ of degree $d=\sum\limits_{j=1}^km_j$ such that every solution of \eqref{1a} has the form $Q\big(z\frac{{\rm d}}{{\rm d}z}\big)F(z) $, where $F$ is a solution of the hypergeometric equation
\begin{gather}\label{2a}
F''-\left( \frac\alpha z+\frac{\beta+d}{z-1} \right) F'+\frac{\gamma\delta}{z(z-1)} F = 0.
\end{gather}
In particular, the monodromy group of equation \eqref{1a} coincides with that of the hypergeometric equation~\eqref{2a}.
\end{thm}

\begin{Remark}\label{remark1} Clearly, conditions \eqref{0d} imply that for any non-zero polynomial $p$ of degree at most $d-1$, the functions $p(z)$ and $z^{\alpha+1}p(z)$ are not solutions of the hypergeometric equation~\eqref{2a}.
\end{Remark}

\begin{Remark}\label{remark2} Under conditions \eqref{0d}, for any non-zero polynomial $Q$ of degree at most $d$ and any non-zero solution $F$ of \eqref{2a}, the function $Q\big(z\frac{{\rm d}}{{\rm d}z}\big)F(z)$ is non-zero. Indeed, a solution $F$ of~\eqref{2a} such that $Q\big(z\frac{{\rm d}}{{\rm d}z}\big)F(z)\equiv 0$, has the following form: there are polynomials $p_1$, $p_2$ of degrees at most $d-1$ and constants $A,B$, such that
\begin{alignat*}{3}
& F(z) = p_1(z) + z^\alpha p_2(z)\qquad && \hbox{\rm if $\alpha$ is not an integer},&\\
& F(z) = p_1(z) + z^\alpha p_2(z) (A+B\log z)\qquad && \hbox{\rm if $\alpha$ is a non-negative integer},&\\
& F(z) = z^\alpha p_1(z) + p_2(z) (A+B\log z)\qquad && \hbox{\rm if $\alpha$ is a negative integer}. &
\end{alignat*}
Then in the first case, both $p_1(z)$ and $z^\alpha p_2(z)$ are solutions of equation~\eqref{2a} in contradiction with Remark~\ref{remark1}. In the second case, $z^\alpha p_2(z)$ is a solution of \eqref{2a} if $B\ne 0$, while for $B=0$, $F(z)$ is a polynomial and $ \deg F<d$ since $Q\big(z\frac{{\rm d}}{{\rm d}z}\big)F(z)=0$. Both options contradict Remark~\ref{remark1}. The third case can be worked out similarly.
\end{Remark}

\begin{thm}\label{theorem2}
For every $\alpha$, $\beta$, $\gamma$, $\delta$, $m_1\lc m_k$, $a_1\lc a_k$, there exist at least one and at most $\prod\limits_{j=1}^k(m_j+1)$ equations~\eqref{1a} with apparent singularities at all points $a_j$. For generic $\alpha$, $\beta$, $\gamma$, $\delta$, $a_1\lc a_k$, the equality holds.
\end{thm}

\begin{thm}\label{theorem3} Let positive numbers $(\alpha_1\lc\alpha_n)$ be given, where $\alpha_1$, $\alpha_2$, $\alpha_3$ are not integers, and all combinations
\begin{gather}\label{coaxial}
\alpha_1\pm\alpha_2\pm\alpha_3\quad\mbox{are not integers}
\end{gather}
for any choice of signs, while $\alpha_4\lc\alpha_n$ are integers, $\alpha_j\geq 2$. Let the points $(z_1\lc z_n)$ in~$\bC$ be given. Then the necessary and sufficient condition of existence of a conformal metric of curvature $1$ with conic singularities at $z_j$ with angles $2\pi\alpha_j$ is
\begin{gather}\label{cond}
\cos^2 \pi\alpha_1+\cos^2 \pi\alpha_2+\cos^2 \pi\alpha_3+2(-1)^\sigma \cos\pi\alpha_1 \cos\pi\alpha_2 \cos\pi\alpha_3 < 1,
\end{gather}
where
\begin{gather*}
\sigma =\sum_{j=4}^n(\alpha_j-1).
\end{gather*}
The number of such distinct metrics does not exceed the product $\alpha_4\cdots\alpha_n$, with equality for generic $\alpha_1$, $\alpha_2$, $\alpha_3$ and $z_j$.
\end{thm}

The necessary and sufficient condition in this theorem was found in \cite{EGT} in the special case of symmetric metrics. The case when~\eqref{coaxial} does not hold was called exceptional in~\cite{EGT}. In this case the unitarizable projective monodromy must be co-axial, that is conjugate to a subgroup of the unit circle~\cite{MP}, and generic singularities $z_1,\ldots,z_n$ {\em cannot be assigned} in the co-axial case~\cite{E2}. Moreover, in the coaxial case, existence of one metric implies the existence of a continuous family of them.

Condition \eqref{cond} is a special case of the Mondello and Panov necessary condition on the angles of arbitrary metric of positive curvature with conic singularities and non-coaxial monodromy.

\section{Proof of Theorem~\ref{theorem2}}\label{section3}

It will be convenient to use a different normalization for our differential equation. By a linear-fractional transformation of the independent variable in (\ref{1a}) we place the singular points in $\C$, so that they become
\begin{gather*}
(z_1,\ldots,z_n,\infty),
\end{gather*}
where $n=k+3$, and infinity is an apparent singularity with exponent difference $1$. Then we multiply the dependent variable on some function to shift the exponents and obtain the Riemann scheme
\begin{gather}\label{rs}
\left(\begin{matrix}z_1&z_2&\ldots&z_n&\infty\vspace{1mm}\\
\ds\frac{1+\alpha_1}{2}&\ds\frac{1+\alpha_2}{2}&\ds\ldots&
\ds\frac{1+\alpha_n}{2}& 0\vspace{1mm}\\
\ds\frac{1-\alpha_1}{2}&\ds\frac{1-\alpha_2}{2}&\ds\ldots&
\ds\frac{1-\alpha_n}{2}&-1\end{matrix}\right),
\end{gather}
where $\alpha_1=\alpha+1$, $\alpha_2=\beta+1$, $\alpha_3=m_1+1\lc \alpha_{n-1}=m_k+1$, $\alpha_n=\gamma-\delta$. This transformation changes neither the exponent differences nor the projective monodromy group, and the new equation has the form
\begin{gather}\label{1b}
w''+\sum_{j=1}^n\left(\frac{1-\alpha_j^2}{4(z-z_j)^2} +\frac{\beta_j}{z-z_j}\right)w=0.
\end{gather}
Regularity of the point $\infty$ implies
\begin{gather}\label{1}
\sum_{j=1}^n\beta_j=0,\\
\label{2}
\sum_{j=1}^n\beta_jz_j=\sum_{j=1}^\infty(1-\alpha_j^2)/4,
\end{gather}
and the condition that the formal solution at $\infty$ contains no logarithm is
\begin{gather}\label{3a}
\sum_{j=1}^n\beta_jz_j^2=\sum_{j=1}^n z_j\big(1-\alpha_j^2\big)/2.
\end{gather}
When $\alpha_j$ is an integer, the singularity at $z_j$ may be apparent. The test for the apparent singularity can be found, for example in~\cite{TL}. Write our equation near $z_j$ as
\begin{gather*}
(z-z_j)^2w''+\left(\sum_{k=0}^\infty x_{k,j}(z-z_j)^k\right)w=0.
\end{gather*}
The singularity at $z_j$ is apparent if and only if
$\alpha_j=\ell_j$ is an integer,
\begin{gather*}
x_{0,j}=\big(1-\ell_j^2\big)/4,
\end{gather*}
and
\begin{gather}\label{app}
Y_{\ell_j}(x_{1,j}\lc x_{\ell_j,j})=0,
\end{gather}
where for any $\ell$, the polynomial $Y_\ell$ is defined by
\begin{gather*}
Y_{\ell}(x_1\lc x_\ell)=
\left|\begin{matrix}x_1&1\cdot(1-\ell)&0&\ldots&0\\
x_2&x_1&2\cdot(2-\ell)&\ldots&0\\
\vdots&\vdots&\vdots&\ddots&\vdots\\
x_{\ell-1}&x_{\ell-2}&x_{\ell-3}&\ldots&(\ell-1)\cdot(-1)\\
x_\ell&x_{\ell-1}&x_{\ell-2}&\ldots&x_1
\end{matrix}\right|.
\end{gather*}
The matrix has $x_1$ on the main diagonal, $x_2$ on the next diagonal below the main one, etc., and the given numbers on the next diagonal above the main one.

Notice that the polynomial $Y_{\ell}$ has the form
\begin{gather}\label{app2}
Y_\ell(x_1\lc x_\ell)=x_1^\ell+Y^*_\ell(x_1\lc x_\ell),\qquad \deg Y^*_\ell< \ell.
\end{gather}
\begin{lem}\label{lem1}
Suppose that all singularities $z_j$ and exponents $\alpha_j$ are fixed, and that $\alpha_3\lc\alpha_{n-1}$ are integers, $\alpha_j=\ell_j$. Then the number of such equations with apparent singularities at $z_3\lc z_{n-1}$ is finite.
\end{lem}

\begin{proof} We take $\beta_3\lc\beta_{n-1}$ as variables, and express $\beta_1$, $\beta_2$, $\beta_n$ in terms of these variables, using \eqref{1}--\eqref{3a}. (The determinant of this $3\times3$ system with respect to $\beta_1$, $\beta_2$, $\beta_n$ is the Vandermonde determinant, so it is not zero.) Then for each $j\in[3,n-1]$ the variables $x_{k,j}$ become linear functions of $\beta_3\lc\beta_{n-1}$. Noticing that $x_{1,j}=\beta_j$, we obtain that in terms of $\beta_j$, $3\leq\beta_j\leq n-1$, our equations \eqref{app}, \eqref{app2} have the form
\begin{gather}\label{equa}
\beta_j^{\ell_j}=S_j(\beta_3\lc\beta_{n-1}),\qquad 3\leq j\leq n-1,
\end{gather}
where $S_j$ are polynomials of degrees at most $\ell_j-1$. Thus
\begin{gather*}
|\beta_j|\leq C\bigl(1+\|\beta\|^{1-1/\ell_j}\bigr)\leq C_1\bigl(1+\|\beta\|^{1-1/\ell}\bigr),\qquad \ell=\max_j\ell_j,
\end{gather*}
where $\|\beta\|^2:=\beta_3^2+\cdots+\beta_{n-1}^2$. Hence, $\|\beta\|\leq C_2 $. As a bounded algebraic set must be finite, this proves the lemma.

So the condition that singularities $z_3\lc z_{n-1}$ of equation \eqref{1b} are apparent is expressed as $n-3$ equations \eqref{equa} with $n-3$ unknowns $\beta_3\lc\beta_{n-1}$. Adding one extra variable $\beta_0$, we make these equations homogeneous:
\begin{gather*}
\beta_j^{\ell_j}=\beta_0 \widetilde S_j(\beta_0,\beta_3\lc\beta_{n-1}),\qquad 3\leq j\leq n-1,
\end{gather*}
with some homogeneous polynomials $\widetilde S$ of degree $\ell_j-1$. Evidently no non-zero solution can have $\beta_0=0$, so Bezout's theorem \cite[Chapter~IV, Section~1]{Sh} implies that the number of solutions
of~\eqref{equa}, counting multiplicity, equals the product of degrees $\ell_3\cdots\ell_{n-1}=\alpha_3\cdots\alpha_{n-1}$.

Now we show that for generic $\alpha_1$, $\alpha_2$, $\alpha_n $, solutions of \eqref{equa} are simple. To do this, we specify $\alpha_1$, $\alpha_2$, $\alpha_n$ to be large integers, and show that in this case, there are {\em at least} $\alpha_3\cdots\alpha_{n-1}$ solutions. It will follow that for generic $\alpha_1$, $\alpha_2$, $\alpha_n $, solutions are of multiplicity~$1$, because the set of all large integers is Zariski dense.

Solutions of \eqref{equa} correspond to equations \eqref{1b} with trivial local projective monodromy around $z_3\lc z_{n-1}$. The number of such equa\-tions~\eqref{1b} is at least the number of equations~\eqref{1b} with trivial projective monodromy around {\em all} singularities, and ratios of solutions of such equations are rational functions with critical points at $z_1\lc z_n$ of multiplicities $\alpha_j-1$ at $z_j$. Rational functions obtained from the same differential equation are {\em equivalent} in the sense that they are obtained from each other by linear-fractional transformations.
\end{proof}

So we prove

\begin{lem}\label{lem2}
Suppose that
\begin{gather}\label{conds}
\alpha_1+\alpha_n\geq (1/2)\sum_{j=1}^n(\alpha_j-1)+2,\qquad \alpha_2\geq\sum_{j=3}^{n-1}(\alpha_j-1)+1.
\end{gather}
Then the number of equivalence classes of rational functions with critical points at generic points~$z_j$ of multiplicities $\alpha_j-1$, for $1\leq j\leq n$ is the product $\alpha_3\cdots\alpha_{n-1}$.
\end{lem}

As a corollary we obtain that trivial local monodromy around $z_4\lc z_n$ implies that the whole projective monodromy group is trivial, assuming that $\alpha_1$, $\alpha_2$, $\alpha_3$ are large enough in comparison with
other~$\alpha_j$.

\begin{proof}[Proof of Lemma~\ref{lem2}]It is known \cite{EGSV} that classes of rational functions with prescribed generic critical points are enumerated by the Young tableaux of shape $2\times (d-1)$, where
\begin{gather*}
d=\frac12 \biggl(\sum_{j=1}^n(\alpha_j-1)+2\biggr)
\end{gather*}
is the degree of the rational function.

The diagrams are filled with $\alpha_j-1$ numbers $j$ by the usual rules: the entries do not decrease in rows (left to right) and strictly increase in columns. It follows from this rule that $1$'s can be only in the first row on the left, while $n$'s in the second row in the right. Condition~\eqref{conds} implies that two entries other than $1$, $n$ cannot stand in the same column: they occupy the place in the first row on the right, and in the second row on the left. There is no other restriction on placement of these entries (from $2$ to $n-1$), except that they must be non-decreasing in their rows. So the tableau is completely determined by specifying the list of entries from $2$ to $n-1$ in the first row. The possible number of $k$'s ($3\leq k\leq n-1$) in the first row is $\alpha_k$, and the $2$'s occupy the remaining places. The second assumption of the lemma ensures that there are always enough of $2$'s to fill the remaining places. So we have $\alpha_3\cdots\alpha_{n-1}$ possibilities.

This proves Lemma~\ref{lem2} and completes the proof of Theorem~\ref{theorem2}.
\end{proof}

\section{Preliminaries on difference equations}\label{section4}

In the proof of Theorem~\ref{theorem1} we use the tools developed in \cite{MTV} which are explained in this section. Our exposition is formally independent of~\cite{MTV}.

Denote by $\tau$ be the shift operator, $(\tau f)(x)=f(x+1)$ acting on meromorphic functions on the real line. For functions $u_1(x)\lc u_k(x)$, define their Casorati determinant (discrete Wronskian) by
\begin{gather*}
\Wr_k[u_1\lc u_k] = \det\big(\tau^{j-1}u_i\big)_{i,j=1}^k.
\end{gather*}
By convention, $\Wr_0=1$.
\begin{lem}\label{L1}
For given functions $u_1(x)\lc u_k(x)$ such that $\Wr_k[u_1\lc u_k]\ne 0$ and a non-zero function $F(x)$, there exists a unique difference operator of the form
\begin{gather}\label{dif}
D_{u_1\lc u_k;F}:=\tau^{k+1}+\sum_{j=1}^k K_j(x)\tau^j+F(x)
\end{gather}
such that $D_{u_1\lc u_k;F}u_i=0$ for every $i=1\lc k$.
\end{lem}

\begin{proof} For given $u_1\lc u_k$ and $F$, consider a system of linear equations
\begin{gather*}
\sum_{j=1}^kK_j \tau^ju_i=-\tau^{k+1}u_i-Fu_i,\qquad i=1\lc k,
\end{gather*}
for the coefficients $K_1\lc K_k$ of $D_{u_1\lc u_k;F}$. The determinant of this linear system equals $\tau\bigl(\Wr_k[u_1\lc u_k]\bigr)\ne 0$, so the system has a unique solution.
\end{proof}

\begin{lem}\label{L2}
For every function $f(x)$, we have
\begin{gather}
D_{u_1\lc u_k; F} f = \frac{\tau\bigl(\Wr_{k+1}[u_1\lc u_k,f]\bigr)+ (-1)^kF\Wr_{k+1}[u_1\lc u_k,f]}{\tau\bigl(\Wr_k[u_1\lc u_k]\bigr)}.\label{e1}
\end{gather}
\end{lem}

\begin{proof} Denote by $D$ the difference operator whose action on $f$ is given by the right-hand side of \eqref{e1}. Then $Du_i=0$ for every $i=1\lc k$, and $D$ is of the form \eqref{dif} with some coefficients $K_1\lc K_k$. Hence, $D=D_{u_1\lc u_k; F}$, by Lemma~\ref{L1}.
\end{proof}

\begin{lem} \label{L3}$ D_{u_1\lc u_k; F}=(\tau-g_{k+1}) (\tau-g_k) \cdots (\tau-g_1)$, where
\begin{gather*}
g_i = \frac{\tau f_i}{f_i},\qquad f_i = \frac{\Wr_i[u_1\lc u_i]}{\Wr_{i-1}[u_1\lc u_{i-1}]},
\qquad i=1\lc k,
\end{gather*}
and
\begin{gather*}
g_{k+1} = \frac{(-1)^{k+1}F\Wr_k[u_1\lc u_k]}
{\tau\bigl(\Wr_k[u_1\lc u_k]\bigr)}.
\end{gather*}
\end{lem}
\begin{proof}Set $F_i=(-1)^i \tau\bigl(\Wr_i[u_1\lc u_i]\bigr)\big/ \Wr_i[u_1\lc u_i]$ for $i=1\lc k$, and $F_{k+1}=F$. Let $D_i=(\tau-g_i) \dots (\tau-g_1)$. We will prove by induction on $i$, that $D_i=D_{u_1\lc u_{i-1}; F_i}$ for any $i=1\lc k+1$.

The base of induction for $i=1$ is immediate. Clearly, $D_i=\tau^i+\sum\limits_{j=1}^{i-1} K_{i,j} \tau^j+F_i$ for some coefficients $K_{i,1}\lc K_{i,i-1}$. Thus by Lemma~\ref{L1}, it remains to show that $D_iu_j=0$ for any $j=1\lc i-1$. Since $D_i=(\tau-g_i) D_{i-1}$, it suffices to show that $D_{i-1}u_j=0$ for any $j=1\lc i-1$. This follows from Lemma~\ref{L2} since $D_{i-1}=D_{u_1\lc u_{i-2}; F_{i-1}}$ by the induction assumption.
\end{proof}

\begin{lem}\label{L4} Let $u_1(x)\lc u_k(x)$ be solutions of the difference equation
\begin{gather}
\sum_{i=0}^{k+1} A_i(x) u(x+i) = 0,\qquad A_0\ne 0,\qquad A_{k+1}\ne 0.\label{e2}
\end{gather}
Let $w(x)$ be a non-zero solution of the difference equation $A_0(x) w(x)=(-1)^{k+1}A_{k+1}(x) w(x+1)$. Then the function
\begin{gather*}
v(x) = \frac{\Wr_k[u_1\lc u_k](x+1)}{A_0(x) w(x)}
\end{gather*}
satisfies the difference equation
\begin{gather}
\sum_{i=0}^{k+1} A_i(x-i) v(x-i) = 0.\label{e4}
\end{gather}
\end{lem}

\begin{proof} If $\Wr_k[u_1\lc u_k]=0$, the claim is trivial. If $\Wr_k[u_1\lc u_k]\ne 0$, consider the difference operator $D=D_{u_1\lc u_k; K_0}$ of the form \eqref{dif} with $K_i=A_i/A_{k+1}$, $i=0\lc k$. Then the difference equation \eqref{e2} reads $A_{k+1} Du=0$.

Define the difference operator $D^\dag=\tau^{-k-1}+\sum\limits_{i=0}^kK_i(x-i) \tau^{-i}$. Then the difference equa\-tion~\eqref{e4} reads $D^\dag(A_{k+1}v)=0$.

The operator $D^\dag$ is obtained from $D$ by applying the formal conjugation antiautomorphism that sends $\tau$ to $\tau^{-1}$ and the multiplication by $x$ to itself. Thus Lemma~\ref{L3} yields
\begin{gather*}
D^\dag= \big(\tau^{-1}-g_1\big) \cdots \big(\tau^{-1}-g_k\big) \big(\tau^{-1}-g_{k+1}\big),
\end{gather*}
where one takes $F=A_0/A_{k+1}$. Straightforwardly, $(\tau^{-1}-g_{k+1}) (A_{k+1}v)=0$, which proves Lemma~\ref{L4}.
 \end{proof}

Now we consider our equation \eqref{1a}. Define the polynomial $A$ and numbers $b_{i,2}$ by
\begin{gather}\label{A}
A(z)=(z-1)\prod_{i=1}^k(z-a_i)= z^{k+1}+\sum_{i=1}^k b_{i,2}z^i+(-1)^{k+1}a_1\cdots a_k,
\end{gather}
and
\begin{gather*}
W(z)=z^\alpha(z-1)^\beta\prod_{i=1}^k(z-a_i)^{m_i}.
\end{gather*}
Multiply equation \eqref{1a} by $z^2A(z)$, so that it becomes
\begin{gather}
A(z) \left(z\frac{{\rm d}}{{\rm d}z}\right)^{ 2}y-B(z) z\frac{{\rm d}y}{{\rm d}z}+C(z)y=0,\label{3}
\end{gather}
where
\begin{gather}
B(z) = A(z)\bigl(1+zW'(z)/W(z)\bigr)\nonumber\\
\hphantom{B(z)}{} = (\gamma+\delta)z^{k+1}-\sum_{i=1}^k b_{i,1}z^i+(-1)^{k+1}(\alpha+1) a_1\cdots a_k,\label{B2}
\end{gather}
with some numbers $b_{i,1}$, and $C(z)$ is a polynomial of degree $k+1$ such that $C(0)=0$.

Let $x$ be a new variable and $\tau$ be the shift: $(\tau u)(x)=u(x+1)$. The difference equation
\begin{gather}
\left(x^2A(\tau)-xB(\tau)+C(\tau)\right) u(x)=0\label{4}
\end{gather}
is called the {\it bispectral dual} of equation \eqref{3}. Bispectral dual equations are discussed in Theorems~4.1 and~4.2 of \cite{MTV}.

Write \eqref{4} in the form
\begin{gather}
(-1)^{k+1}a_1\cdots a_k x(x-\alpha-1) u(x) +\sum_{i=1}^k b_i(x) u(x+i) \nonumber\\
\qquad{} +(x-\gamma)(x-\delta) u(x+k+1) = 0, \label{5}
\end{gather}
where $b_1\lc b_k$ are quadratic polynomials
\begin{gather*}
b_i(x)=b_{i,2}x^2+b_{i,1}x+b_{i,0},
\end{gather*}
the constants $b_{i,1}$ and $b_{i,2}$ being defined in \eqref{B2} and~\eqref{A}, respectively.

The equation
\begin{gather}
\big(A\big(\tau^{-1}\big) x^2-B\big(\tau^{-1}\big) x+C\big(\tau^{-1}\big)\big)v(x)=0\label{6}
\end{gather}
is called the {\it formal conjugate} to equation \eqref{4}. Equivalently, \eqref{6} can be written as
\begin{gather}
(-1)^{k+1}a_1\cdots a_k x(x-\alpha-1) v(x) +\sum_{i=1}^k b_i(x-i) v(x-i)\nonumber\\
\qquad{} + (x-k-1-\gamma)(x-k-1-\delta) v(x-k-1) = 0. \label{7}
\end{gather}
If $u_1\lc u_k$ are solutions of equation \eqref{5}, then by Lemma~\ref{L4}
\begin{gather}
v(x) = (a_1\cdots a_k)^{-x-1} \frac{\Gamma(x-\gamma) \Gamma(x-\delta)}{\Gamma(x+1) \Gamma(x-\alpha)}\det\bigl(u_i(x+j)\bigr)_{i,j=1}^k\label{8}
\end{gather}
is a solution of equation \eqref{7}.

Equations \eqref{7} and \eqref{1a} are related as follows: if $v(x)$ solves \eqref{7} and the series
\begin{gather}
f(z) = \sum_{n=-\infty}^\infty v(n+\xi) z^{n+\xi}\label{8a}
\end{gather}
converges for some $\xi $, then $f(z)$ is a solution of equation \eqref{1a}.

\section{Proof of Theorem~\ref{theorem1}}\label{section5}

Let $y(z)$ be a solution of equation \eqref{1a}. Define functions
\begin{gather}
 u_i(x) = \oint_{a_i}\frac{y(z) z^{x-\alpha-2} {\rm d}z}
{(z-1)^{\beta+1}\prod\limits_{j=1}^k(z-a_j)^{m_j+1}}
= 2\pi\sqrt{-1} \Res_{z=a_i} \frac{y(z) z^{x-\alpha-2}}
{(z-1)^{\beta+1}\prod\limits_{j=1}^k(z-a_j)^{m_j+1}},\label{9}
\end{gather}
where the integral is over a small circle around $a_i$ counterclockwise. Clearly, $u_i(x)=a_i^xp_i(x)$, where $p_i(x)$ is a polynomial of degree $m_i$. Since $y(z)$ solves equation~\eqref{3}, the function
\begin{gather}\label{yt}
\yt(z) =
\frac{y(z)}{z^{\alpha+1}(z-1)^{\beta+1}\prod\limits_{j=1}^k(z-a_j)^{m_j+1}} =\frac{y(z)}{zA(z)W(z)}
\end{gather}
satisfies the differential equation
\begin{gather}
\left(z\frac{{\rm d}}{{\rm d}z}\right)^{ 2} \bigl( A(z) \yt\bigr)+ z\frac{{\rm d}}{{\rm d}z}\bigl(B(z) \yt)+C(z) \yt = 0.\label{3t}
\end{gather}
By formulas \eqref{9}, \eqref{yt},{\samepage
\begin{gather*}
u_i(x)=\oint_{a_i}\yt(z) z^{x-1} {\rm d}z,
\end{gather*}
and equation \eqref{3t} yields that $u_1\lc u_k$ are solutions of equation \eqref{4}.}

Use functions \eqref{9} in formula \eqref{8} for a solution $v(x)$ of equation~\eqref{7}. Then
\begin{gather}
v(x) = \frac{\Gamma(x-\gamma) \Gamma(x-\delta)}{\Gamma(x+1) \Gamma(x-\alpha)} Q(x),\label{10}
\end{gather}
where
\begin{gather*}
Q(x) = (a_1\cdots a_k)^{-x-1}\det\bigl(u_i(x+j)\bigr)_{i,j=1}^k=
 \det\bigl(a_i^{j-1}p_i(x+j)\bigr)_{i,j=1}^k
\end{gather*}
is a non-zero polynomial of degree $d=m_1+\dots+m_k$.

To prove Theorem~\ref{theorem1}, it suffices to show that for any solution $F(z)$ of \eqref{2a}, the function $Q\big(z\frac{{\rm d}}{{\rm d}z}\big)F(z)$ solves \eqref{1a} because under conditions \eqref{0d} the function $Q\big(z\frac{{\rm d}}{{\rm d}z}\big)F(z)$ is non-zero, see Remark~\ref{remark2} after Theorem~\ref{theorem1}. Moreover, it is enough to prove this statement for generic values of $\alpha$, $\gamma$, $\delta$. Indeed, by the proof of Theorem~\ref{theorem2}, equation~\eqref{1a} can be deformed to generic values of $\alpha$, $\beta$, $\gamma$, $\delta$ without moving the singularities at $a_1\lc a_k$ and keeping all of them apparent. To make $F$ and $Q$ change continuously under the deformation, one can parametrize~$F$ and a~solution~$y(z)$ of~\eqref{1a} defining~$Q$ by their data at a fixed regular point.

For generic $\alpha$, $\gamma$, $\delta$, the functions
\begin{gather*}
F_1(z) = \sum_{n=0}^\infty \frac{\Gamma(n-\gamma) \Gamma(n-\delta)}{\Gamma(n+1) \Gamma(n-\alpha)} z^n = \frac{\Gamma(-\alpha)}{\Gamma(-\gamma) \Gamma(-\delta)} \,{}_2F_1(-\gamma,-\delta;-\alpha;z)
\end{gather*}
and
\begin{gather*}
F_2(z)= \sum_{n=0}^\infty
\frac{\Gamma(n+1+\alpha-\gamma) \Gamma(n+1+\alpha-\delta)} {\Gamma(n+2+\alpha) \Gamma(n+1)} z^{n+1+\alpha}\\
\hphantom{F_2(z)}{} = \frac{\Gamma(2+\alpha)} {\Gamma(1+\alpha-\gamma) \Gamma(1+\alpha-\delta)} z^{1+\alpha} \, {}_2F_1(1+\alpha-\gamma,1+\alpha-\delta;2+\alpha;z)
\end{gather*}
are independent solutions of the hypergeometric equation \eqref{2a}, see for example \cite{Mellin}. The functions
\begin{gather*}
f_1(z) = Q\left({z\frac{{\rm d}}{{\rm d}z}}\right) F_1(z) = \sum_{n=0}^\infty \frac{\Gamma(n-\gamma) \Gamma(n-\delta)}{\Gamma(n+1) \Gamma(n-\alpha)} Q(n) z^n,\\
f_2(z) = Q\left({z\frac{{\rm d}}{{\rm d}z}}\right) F_2(z) = \sum_{n=0}^\infty\frac{\Gamma(n+1+\alpha-\gamma) \Gamma(n+1+\alpha-\delta)}{\Gamma(n+2+\alpha) \Gamma(n+1)} Q(n+1+\alpha) z^{n+1+\alpha},
\end{gather*}
coincide with the series~\eqref{8a} for function (\ref{10}) and $\xi=0$ or $\xi=1+\alpha$. Since the series converge for $|z|<1$, both $f_1(z)$ and~$f_2(z)$ are solutions of~(\ref{1a}).

This proves Theorem~\ref{theorem1}.

\section{Application to metrics of positive curvature}\label{section6}

Derivation of Theorem~\ref{theorem3} from Theorems~\ref{theorem1} and~\ref{theorem2} consists of standard arguments which we outline here.

Suppose that we have a Riemannian metric of curvature $1$ on the punctured sphere $S\backslash\{ z_1\lc z_n\}$, and conic singularities at~$z_j$. Conic singularity at~$z_j$ means that a neighborhood of~$z_j$ is isometric to a cone with intrinsic angle $2\pi\alpha_j>0$ at the vertex, and curvature~$1$ at all other points. It is well-known that a surface of constant curvature $1$ is locally isometric to the standard unit sphere. Consider this local isometry $f$ in a neighborhood of a non-singular point. As~$f$ is an isometry, it is conformal, thus analytic, and admits an analytic continuation along any curve which does not pass through the singularities. We obtain a multi-valued {\em developing map}
\begin{gather*}
f\colon \ S\backslash\{ z_1\lc z_n\}\to\bC.
\end{gather*}
Developing map extends continuously (in the sense of radial limits) to the singularities. As the sphere has only one conformal structure, $S$ is also the Riemann sphere. The monodromy of our map consists of rotations of the sphere which can be represented by fractional linear transformations in ${\rm PSU}(2)$. The Schwarzian derivative
\begin{gather*}
\{ f,z\}:=\frac{f'''}{f'}-\frac{3}{2}\left(\frac{f''}{f'}\right)^2
\end{gather*}
remains unchanged when $f$ undergoes a fractional-linear transformation, so it is a single-valued analytic map from $S\backslash\{z_1\lc z_n\}$ to $\bC$. Consideration near a singularity shows that $f(z)=(c+o(1))z^\alpha$ in the local coordinate with angle $2\pi\alpha$. Computing the Schwarzian we obtain that it has a double pole with the principal term
\begin{gather*}
\frac{1-\alpha^2}{2z^2},
\end{gather*}
so we have a third order differential equation for $f$:
\begin{gather}\label{sch}
\{ f,z\}=\sum_{j=1}^n\frac{1-\alpha_j^2}{2(z-z_j)^2}+2\frac{\beta_j}{z-z_j}.
\end{gather}
When $f$ is a developing map, the monodromy of this differential equation must be unitarizable, which means that there is at least one solution with monodromy in ${\rm PSU}(2)$. Developing maps correspond to the
same metric if they are related by $f_1=L\circ f_2$ where $L\in {\rm PSU}(2)$. The general solution of \eqref{sch} consists of fractional-linear transformations of one solution.

So different solutions of the same equation can define different metrics if they are related by $f_1=L\circ f_2$ where $L$ satisfies the condition that $\phi^{-1}L\phi$ is unitary for all $\phi$ from the projective monodromy group. This is only possible when $L$ itself is in ${\rm PSU}(2)$ or when the monodromy is coaxial.

Finally we mention that the general solution of \eqref{sch} is the ratio of two linearly independent solutions of \eqref{1b}.

\begin{proof}[Proof of Theorem~\ref{theorem3}] Suppose that condition \eqref{coaxial} holds. Then Theorem~\ref{theorem1} is applicable. By Theorem~\ref{theorem1}, equations \eqref{1a} and \eqref{2a} have the same monodromy. Monodromy of the hypergeometric equation \eqref{2a} depends only on the exponents, and under the assumption \eqref{coaxial}, it is unitarizable if and only if the exponents satisfy \eqref{cond}, as it is proved in \cite{E, FKKRUY}. So condition~\eqref{cond} is necessary and sufficient for unitarizability
of the monodromy when \eqref{coaxial} is satisfied.

Now if \eqref{cond} holds, the monodromy is not co-axial, so every
equation \eqref{1a} with apparent singularities and unitary monodromy defines exactly one metric. So the statement on the number of the metrics follows from Theorem~\ref{theorem2}.
\end{proof}

\begin{Remark}
A referee suggested another proof of Theorem~\ref{theorem3} without using Theorem~\ref{theorem1}. Suppose that $A_1$, $A_2$, $A_3$ are matrices in ${\rm SL}(2)$ which generate a non-commutative group and satisfy $A_1A_2A_3=I$ and $t_j:=\tr A_j\in(-2,2)$. Then a necessary and sufficient condition of simultaneous unitarizability of these matrices is{\samepage
\begin{gather}\label{u}
t_1^2+t_2^2+t_3^2-t_1t_2t_3<4.
\end{gather}
This statement is well-known and easy to prove, see, for example \cite[Theorem~3.7]{BS}.}

Consider now an equation with the Riemann scheme~(\ref{rs}) where $\alpha_1$, $\alpha_2$, $\alpha_3$ are not integers while all other $\alpha_j$ are integer and the singularities at $z_j$, $4\leq j\leq n$, and infinity are apparent. Let~$M_j$ be the monodromy matrices. Then it is easy to see that $\det M_j=1$ for all singularities, $M_j=(-I)^{\alpha_j-1}$ for all apparent singularities including infinity, where $\alpha_\infty=1$, and $\tr M_i=-\cos\pi\alpha_i$ for $i=1,2,3$. Thus $M_1M_2M_3=(-I)^{\sigma-1}$, and applying condition~\eqref{u} to the matrices $A_1=M_1$, $A_2=M_2$, $A_3=(-1)^{\sigma-1}M_3$, we obtain a necessary and sufficient condition of unitarizability, which is nothing but~\eqref{cond}.
\end{Remark}

\subsection*{Acknowledgments}

A.~Eremenko was supported by NSF grant DMS-1665115. V.~Tarasov was supported in part by Simons Foundation grant 430235. We thank Andrei Gabrielov for illuminating discussions of this paper and the referees whose remarks improved the exposition.

\pdfbookmark[1]{References}{ref}
\LastPageEnding


\begin{thebibliography}{99}
\footnotesize\itemsep=0pt

\bibitem{BS}
Buckman R., Schmitt N., Spherical polygons and unitarization, {u}npublished,
 available at \url{http://www.gang.umass.edu/reu/2002/gon.pdf}.

\bibitem{TL}
Cui G., Gao Y., Rugh H.H., Tan L., Rational maps as {S}chwarzian primitives,
 \href{https://doi.org/10.1007/s11425-016-5140-7}{\textit{Sci. China Math.}} \textbf{59} (2016), 1267--1284,
 \href{https://arxiv.org/abs/1511.04246}{arXiv:1511.04246}.

\bibitem{E}
Eremenko A., Metrics of positive curvature with conic singularities on the
 sphere, \href{https://doi.org/10.1090/S0002-9939-04-07439-8}{\textit{Proc. Amer. Math. Soc.}} \textbf{132} (2004), 3349--3355,
 \href{https://arxiv.org/abs/math.MG/0208025}{math.MG/0208025}.

\bibitem{E2}
Eremenko A., Co-axial monodromy, \href{https://doi.org/10.2422/2036-2145.201706_022}{\textit{Ann. Sc. Norm. Super. Pisa}}, {t}o
 appear, \href{https://arxiv.org/abs/1706.04608}{arXiv:1706.04608}.

\bibitem{EGSV}
Eremenko A., Gabrielov A., Shapiro M., Vainshtein A., Rational functions and
 real {S}chubert calculus, \href{https://doi.org/10.1090/S0002-9939-05-08048-2}{\textit{Proc. Amer. Math. Soc.}} \textbf{134}
 (2006), 949--957, \href{https://arxiv.org/abs/math.AG/0407408}{math.AG/0407408}.

\bibitem{EGT1}
Eremenko A., Gabrielov A., Tarasov V., Metrics with conic singularities and
 spherical polygons, \textit{Illinois~J. Math.} \textbf{58} (2014), 739--755,
 \href{https://arxiv.org/abs/1405.1738}{arXiv:1405.1738}.

\bibitem{EGT}
Eremenko A., Gabrielov A., Tarasov V., Spherical quadrilaterals with three
 non-integer angles, \href{https://doi.org/10.15407/mag12.02.134}{\textit{J.~Math. Phys. Anal. Geom.}} \textbf{12} (2016),
 134--167, \href{https://arxiv.org/abs/1504.02928}{arXiv:1504.02928}.

\bibitem{FKKRUY}
Fujimori S., Kawakami Y., Kokubu M., Rossman W., Umehara M., Yamada K., C{MC}-1
 trinoids in hyperbolic 3-space and metrics of constant curvature one with
 conical singularities on the 2-sphere, \href{https://doi.org/10.3792/pjaa.87.144}{\textit{Proc. Japan Acad. Ser.~A Math.
 Sci.}} \textbf{87} (2011), 144--149, \href{https://arxiv.org/abs/1008.3734}{arXiv:1008.3734}.

\bibitem{H}
Heins M., On a class of conformal metrics, \href{https://doi.org/10.1017/S002776300002376X}{\textit{Nagoya Math.~J.}} \textbf{21}
 (1962), 1--60.

\bibitem{Mellin}
Iwasaki K., Kimura H., Shimomura S., Yoshida M., From {G}auss to {P}ainlev\'e.
 A modern theory of special functions, \href{https://doi.org/10.1007/978-3-322-90163-7}{\textit{Aspects of Mathematics}}, Vol.~E16, Friedr. Vieweg \& Sohn, Braunschweig, 1991.

\bibitem{K}
Klein F., Mathematisches Seminar zu G\"ottingen, Winter-Semester 1905/06,
 {E}nglish transl. available at
 \url{http://www.claymath.org/publications/klein-protokolle}.

\bibitem{MP}
Mondello G., Panov D., Spherical metrics with conical singularities on a
 2-sphere: angle constraints, \href{https://doi.org/10.1093/imrn/rnv300}{\textit{Int. Math. Res. Not.}} \textbf{2016}
 (2016), 4937--4995, \href{https://arxiv.org/abs/1505.01994}{arXiv:1505.01994}.

\bibitem{MTV}
Mukhin E., Tarasov V., Varchenko A., Bispectral and
 {$({\mathfrak{gl}}_N,{\mathfrak{gl}}_M)$} dualities, discrete versus
 differential, \href{https://doi.org/10.1016/j.aim.2007.11.022}{\textit{Adv. Math.}} \textbf{218} (2008), 216--265,
 \href{https://arxiv.org/abs/math.QA/0605172}{math.QA/0605172}.

\bibitem{P}
Picard E., De l'int\'egration de l'\'equation {$\Delta u=e^u$} sur une surface
 de {R}iemann ferm\'ee, \href{https://doi.org/10.1515/crll.1905.130.243}{\textit{J.~Reine Angew. Math.}} \textbf{130} (1905),
 243--258.

\bibitem{Sch}
Scherbak I., Rational functions with prescribed critical points, \href{https://doi.org/10.1007/s00039-002-1365-4}{\textit{Geom.
 Funct. Anal.}} \textbf{12} (2002), 1365--1380, \href{https://arxiv.org/abs/math.QA/0205168}{math.QA/0205168}.

\bibitem{S}
Schilling F., Ueber die {T}heorie der symmetrischen {$S$}-{F}unctionen mit
 einem einfachen {N}ebenpunkte, \href{https://doi.org/10.1007/BF01453704}{\textit{Math. Ann.}} \textbf{51} (1899),
 481--522.

\bibitem{Sh}
Shafarevich I.R., Basic algebraic geometry. I.~Varieties in projective space,
 2nd~ed., \href{https://doi.org/10.1007/978-3-642-57908-0}{Springer-Verlag}, Berlin, 1994.

\bibitem{T}
Troyanov M., Prescribing curvature on compact surfaces with conical
 singularities, \href{https://doi.org/10.2307/2001742}{\textit{Trans. Amer. Math. Soc.}} \textbf{324} (1991),
 793--821.

\end{thebibliography}
\end{document}